\documentclass[11pt]{article}
\usepackage{amsmath,amssymb, amsthm,delarray}
\usepackage{graphicx}
\numberwithin{equation}{section}
\newtheorem{thm}{Theorem}[section]
\newtheorem{lemma}[thm]{Lemma}
\newtheorem{remark}[thm]{Remark}
\newtheorem{corollary}[thm]{Corollary}

{\rm}
\newtheorem{example}{Example}{\rm}
{\rm}

\def\beq{\begin{equation} }
\def\eeq{\end{equation} }
\def\e{\hbox{\rm e}}

\def\N{\mathbb{N}}
\def\M{\mathbf{M}}
\def\R{\mathbb{R}}

\def\Q{\mathbf{Q}}
\def\DD{\mathbf{D}}

\def\M{\mathbf{M}}

\def\G{\mathbf{G}}

\def\v{\mathbf{v}}
\def\f{\mathbf{f}}

\def\u{\mathbf{u}}
\def\f{\mathbf{f}}

\def\x{\mathbf{x}}

\def\y{\mathbf{y}}

\def\e{\mathbf{e}}

\def\R{\mathbb{R}}

\def\X{\mathbf{X}}
\def\g{\mathbf{g}}

\begin{document}
\title{Reconstruction of algebraic-exponential data from  moments}
\author{Jean-Bernard Lasserre\thanks{LAAS-CNRS and Institute of Mathematics, University of Toulouse, 7 Avenue du Colonel Roche,
BP 54200, 31031 Toulouse cedex 4, France; email: {\tt lasserre@laas.fr}}
\and Mihai Putinar\thanks{
Division of Mathematical Sciences, Nanyang Technological University, Singapore 637371, and
Department of Mathematics, University of California, Santa Barbara, CA 93106;
email: {\tt mputinar@ntu.edu.sg}, {\tt mputinar@math.ucsb.edu}}}

\maketitle
\begin{abstract}
Let $\G$ be a bounded open subset of Euclidean space
with real algebraic boundary $\Gamma$.
Under the assumption that the degree $d$ of $\Gamma$ is given, and the power moments of
the Lebesgue measure on $\G$ are known up to order $3d$, we describe an algorithmic procedure
for obtaining a polynomial vanishing on $\Gamma$. The particular case of semi-algebraic sets defined by a
single polynomial inequality raises an intriguing question related to the finite determinateness of the full moment sequence. The more general case of a measure with density equal to the exponential of a polynomial is treated in parallel. Our approach 
relies on Stokes theorem and simple Hankel-type matrix identities.
\end{abstract}


.

\section{Introduction}

The present paper is concerned with the exact recovery of a semi-algebraic set $\G$ in Euclidean space from power moments of the Lebesgue measure with an exponential of a polynomial as a density.  Regarded as a rather specialized inverse problem the reconstruction algorithm proposed below is a part of
current studies in geometric tomography, computerized tomography, and in particular shape recognition and shape recovery. We derive with minimal technical means a series of simple observations about the exact reconstruction from moments of various algebraic/exponential
data. The matrix analysis framework we propose below is an extension of sums of squares and semi-definite programing techniques recently developed in polynomial optimization
\cite{lasserre}. 

Reconstruction algorithms of particular shapes abound: polyhedra \cite{milanfar,gravin}, planar quadrature domains \cite{Putinar,Ebenfelt}, convex bodies \cite{Gardner}, sublevel sets of homogeneous polynomials \cite{lasserre-dcg}. 
It is not our aim to comment or compare them, nor to dwell into the long and glorious past of the inversion of algebraic integral transforms \cite{AGV,Pham,Palamodov}. Central to all these studies is the structure of moments
of algebraic data, again a rich and very ramified topics with old roots \cite{Krein,KS}
and current contributions \cite{morosov1,morosov2}.

The contents is the following. We first consider the case of bounded open set $\G\subset\R^n$ with algebraic boundary $\partial\G$.
If the degree $d$ of $\partial\G$ and moments (up to order $3d$)
of the Lebesgue measure on $\G$ are known then the vector of coefficients $\g$
of a polynomial $g$ of degree $d$ that vanishes on $\partial\G$ is uniquely determined (up to a constant)
as the generator of the one-dimensional kernel of
a certain moment-like matrix
whose entries are obtained from moments of the Lebesgue measure on $\G$.
That is, only finitely many such moments (up to order $3d$) are needed and computing $\g$ reduces
to a simple linear algebra procedure. Moreover, in case when $\G$ is convex only moments up to order $2d$ suffice.

An important consequence concerns the case of a sublevel set $\G = \{\,\x: g(\x)\leq 0\,\}$ of a polynomial $\g\in\R[\x]_d$.
Indeed the moments of the Lebesgue measure on $\G$ can all be deduced from those up to order
$3d$ (and $2d$ if $\G$ is convex)! That is, exactly as in the classical situation of a degenerated moment problem on the line, we single out a finite determinateness
property of moment sequences attached to algebraic/exponential data. The analogy to the well understood moment rigidity of the Gaussian distribution is striking, although the constructive aspects of this finite determinateness remain
too theoretical in general. To be more precise, we show that for a given polynomial $p$ the moment sequence
\begin{equation}
\label{bounded}
 \int_{g(\x)<1} \x^\alpha \exp(p(\x)) \,d\x, \ \ \alpha \in \N^n,
 \end{equation} is determined by its finite initial segment $|\alpha|<N$,
with $N$ depending only on the degrees of $g$ and $p$. And similarly, under the necessary integrability assumption, the full sequence of moments
$$ \int_{\R^n} \x^\alpha \exp(q(\x))\, d\x, \ \ \alpha \in \N^n,$$ 
is determined by its initial finite section $|\alpha|<N$, where $N$ depends only on the degree of the unknown polynomial
$q$. (Notice that the family of measures $\{\mu_g\}$ indexed by $g\in\R[\x]_d$ and described in (\ref{bounded}), does not form an exponential family.)

Finally, when the boundary $\partial\G$ is not algebraic, we describe a simple heuristic procedure
to compute a polynomial $g$ whose level set $\{\x: g(\x)=0\}$ approximates $\partial\G$, and
the higher the degree of $g$, the better is the approximation. An illustrative simple case of a real analytic boundary shows how this
procedure can be very efficient. The error estimates for this approximation procedure will be discussed in a separate article. In particular, a comparison with the complex orthogonal polynomial reconstruction method is in order \cite{GPSS},
as well as a parallel to the ubiquitous Prony method \cite{PP}.

Guided by simplicity, clarity of exposition and accessibility to non-experts, our
article remains at an elementary level, with precise references to the technical aspects of real algebraic geometry or geometric integration theory needed in the proofs. \\

{\it Acknowledgements.} The first author was partially supported by a grant from the PGMO
program of the {\it Fondation Math\'ematique Jacques Hadamard} (FMJH) and the second author was partially supported by a Nanyang Technological University grant. Both authors are grateful to the Institute of Mathematical Sciences, Singapore for offering an inspiring climate of research
during the special November 2013-January 2014 program devoted to Moment Problems.

\section{Main result}

Let $\R[\x]$ be the ring of polynomials in the variables $\x=(x_1,\ldots,x_n)$ and let
$\R[\x]_d$ be the vector space of polynomials of degree at most $d$
(whose dimension is $s(d):={n+d\choose n}$).
For every $d\in\N$, let  $\N^n_d:=\{\alpha\in\N^n:\vert\alpha\vert \,(=\sum_i\alpha_i)=d\}$, 
and
let $\v_d(\x)=(\x^\alpha)$, $\alpha\in\N^n$, be the vector of monomials of the canonical basis 
$(\x^\alpha)$ of $\R[\x]_{d}$. 
A polynomial $f\in\R[\x]_d$ is written
\[\x\mapsto f(\x)\,=\,\sum_{\alpha\in\N^n}f_\alpha\,\x^\alpha,\]
for some vector of coefficients $\f=(f_\alpha)\in\R^{s(d)}$.

A real-valued polynomial $g:\R^n\to\R$ is homogeneous of degree $d$ ($d\in\N$)
if $g(\lambda\x)=\lambda^dg(\x)$ for all $\lambda$ and all $\x\in\R$. 

For an arbitrary polynomial $g\in\R[\x]_d$, write
\[g(\x)\,=\,\sum_{k=0}^dg_k(\x),\qquad\forall\x\in\R^n,\]
where for each $k\leq d$, $g_k$ is homogeneous of degree $k$.

\subsection{General framework}

\begin{lemma} Let $\G$ be a bounded open subset of $\R^n$ and let $g$ be a polynomial
satisfying $g(x) = 1$ for all $x \in \partial \G$.
Then for every $\alpha\in\N^n$:
\begin{equation}
\label{lem1-1}
\int_{\G}\x^\alpha\,(1-g(\x))\,d\x\,=\,\sum_{k=1}^d
\frac{k}{n+\vert\alpha\vert}\int_{\G}\,\x^\alpha\,g_k(\x)d\x.
\end{equation}
\end{lemma}
\begin{proof}

If the boundary of $\G$ were smooth an application of Stokes theorem would imply the identity in the statement.
Indeed, denote by $\vec{n}_\x$ be the outward pointing normal to $\G$ at the point $\x\in\partial\G$.
With the vector field $X=\x$ and function $f=\x^\alpha (1-g)$, Stokes' formula yields:
\begin{equation}
\label{eq-stokes}
\int_\G{\rm Div}(X) f\,d\x+\int_\G X\cdot f\,d\x \,=\,\int_{\partial\G}\langle X,\vec{n}_\x\rangle\,f d\sigma,\end{equation}
where $\sigma$ is the surface area measure on $\partial\G$.
Therefore (\ref{lem1-1}) follows because $f$ vanishes on $\partial\G$, ${\rm Div}(X)=n$, and 
\[X\cdot f=\vert\alpha\vert f-\x^\alpha\langle\x,\nabla g(\x)\rangle\,=\,
\vert\alpha\vert f-\x^\alpha\sum_{k=1}^d kg_k(\x).\]

In the presence of singularities of $\partial \G$, Whitney's generalization of Stokes theorem \cite{Whitney} Theorem 14A
applies, leading to the same conclusion.
\end{proof}

Next, given a bounded open set $\G \subset \R^n$, let
$\y=(y_\alpha)$, $\alpha\in\N^n$, be the vector of moments of the restriction of the Lebesgue measure to $\G$:
\[y_\alpha\,=\,\int_\G\x^\alpha\,d\x,\qquad\forall\alpha\in\N^n,\]
and let $\y^k=(y_\alpha)$, $\alpha\in\N^n_k$, be the finite vector in $\R^{s(d)}$ of moments up to order $k$.
We define a renormalised moment matrix $\M^d_k(\y)$, $k,d\in\N$, as follows:\\

- $s(d)$ ($={n+d\choose n}$) columns indexed by $\beta\in\N^n_d$,

- rows indexed by $\alpha\in\N^n_k$,
and  with entries:
\begin{equation}
\label{overlineM-1}
\M^d_k(\y)(\alpha,\beta)
\,:=\,\frac{n+\vert\alpha\vert+\vert\beta\vert}{n+\vert\alpha\vert}\,y_{\alpha+\beta},\qquad\alpha\in\N^n_k,\,\beta\in\N^n_d.\end{equation}

Our aim is to reverse the statement of the above Lemma and recover from finitely many moments a defining function of an open set whose boundary is contained in the real zero set of a polynomial. Questions of uniqueness, choice of the coordinate system, irreducibility naturally arise, and we will address them in subsequent corollaries of the following theorem. First we consider the generic case of a distinguished point $\x=0$ not belonging to the Zariski closure of the boundary.

Assume that $\G$ is an open subset of $\R^n$, so that $\G = {\rm int} \overline{\G}$ (that is $\G$ does not contain "slits") and
the boundary $\partial \G$ is real algebraic. The dimension of $\partial \G$ is then $n-1$, so that the ideal associated to it is principal
(see for instance Theorem 4.5.1 in \cite{BCR}). In particular, there exists a polynomial $g$, vanishing of the first order on 
every smooth component of $\partial \G$, with the property that every other polynomial vanishing on $\partial \G$ is a multiple of $g$, in standard algebraic notation
$$ I(\partial \G) = (g).$$
We define the {\it degree} of $\partial \G$ as the degree of the generator $g$ of the ideal $I(\partial \G)$.

 Note however that the polynomial $g$ may
vanish at internal points of $\G$, and it may even change sign there. 
A simple example supporting this assertion can be obtained from the sector of a disk of large inner angle:
$$ \G = \{ \x \in \R^2:\,  x_1^2 + x_2^2 <1 \} \setminus \{ \x \in \R^2:\, 0 \leq |x_2| \leq x_1\}.$$
The defining function of $\partial \G$ is $g(\x) = (1-x_1^2-x_2^2)(x_1-x_2)(x_1+x_2)$, which changes sign inside $\G$.

\begin{thm}
\label{thmain}
Let $\G \subset \R^n$ be a bounded open set with real algebraic boundary. Assume that $\G = {\rm int}\, \overline{\G},$ the boundary $\partial \G$ has degree $d$ and the point $x=0$ does not belong to the zero set of the
ideal $I(\partial \G)$.

Let $\M^d_{2d}(\y)(\alpha,\beta)$ be the kernel defined in (\ref{overlineM-1}) associated with the moments of $\G$. Then the linear system
\begin{equation}
\label{thmain-1}
\M^d_{2d}(\y)\,\left[\begin{array}{c}-1\\ \g\end{array}\right]\,=\,0.
\end{equation}
admits a unique solution $\g\in \R^{s(d)-1}$, and the polynomial $g$ with coefficients $(0,\g)$ satisfies
$$ (\x \in \partial \G) \Rightarrow  (g(\x) =1).$$
\end{thm}

\begin{proof} 
Again, if the boundary of $\G$ were smooth, we could simply remark that
 (\ref{thmain-1}) is just a rephrasing of (\ref{lem1-1}) (for all $\alpha\in\N^n_{2d}$) in terms of the vector $\g$ and the matrix $\M^d_{2d}(\y)$. 

We start by noticing that the algebraic
boundary $\partial \G$ admits a semi-algebraic triangulation (see \cite{BCR} Sections 9.2-3). Denote
$$ \partial \G = Z \cup Z',$$
where $Z$ is a finite union of smooth $(n-1)$-submanifolds of $\R^n$, leaving $\G$ on one side, and
$Z'$ is the union of the lower dimensional strata, so that $Z'$ has vanishing $(n-1)$-measure. In this case
a generalization of Stokes Theorem is valid, for smooth differential forms (see \cite{Whitney} Theorem 14A,
of \cite{federer}). 

According to the above theorem, the system (\ref{thmain-1}) is compatible;
indeed by assumption there is a polynomial, say $g^*\in\R[\x]_d$,
vanishing on $\partial\G$ (hence on $Z$) and with coefficient vector 
of the form $(-1,\g^*)\in\R^{s(d)}$, since $g^*(0)\neq0$.
Therefore  $(-1,\g^*)^T$ is a solution of (\ref{thmain-1}).

Next, let $(-1,\g)$ denote an arbitrary solution of (\ref{thmain-1}) and
let $g\in\R[\x]_d$ be the polynomial having $(0,\g)$ as vector of coefficients
(hence vanishing at $\x=0$). Then we infer by Stokes
Theorem:
\begin{equation}
\label{aux1}
\int_Z\langle X,\vec{n}_\x\rangle\,(1-g)\,\x^\alpha d\sigma\,=\,0,\qquad \forall\alpha\in\N^n_{2d}.
\end{equation}
Assume that the function $\langle X, \vec{n}_\x \rangle$ vanishes on a set $S \subset Z$ of non-null $\sigma$ measure.
Then, there exists a polynomial $h(\x)$ which vanishes on a connected component $Z_1$ of $Z$, so that
$\nabla h$ is not identically zero on $Z_1$. The polynomial function $\x\mapsto\langle \nabla h(\x), \x\rangle $ vanishes on $S$, as
$\nabla h(\x)$ is colinear with the normal vector $\vec{n}_\x$ at $\x$ on the hypersurface $Z_1$. Since $S$ has non-null
$\sigma$-measure, we infer that $\langle \nabla h(\x), \x\rangle $ is identically equal to zero on $Z_1$. In virtue of 
Hilbert's Nullestellensatz (applied to the complexified ring of polynomials), we have
$$ \langle \nabla h(\x), \x\rangle  = \theta(\x) h(\x)$$
where $\theta$ is a complex polynomial. Indeed, fix a smooth point $a$ of $Z_1$ and remark
that the ideal generated by $h$ in the local ring ${\mathcal O}_{a}$ is prime.
By counting degrees, we find that
$$  \langle \nabla h(\x), \x\rangle  = \lambda \,h(\x)$$
where $\lambda$ is a real constant. Consequently  $h(0)=0$, a contradiction.

As a matter of fact, the above argument implies that any polynomial $q$ vanishing on $Z_1$
is a multiple of $h$, hence $q(0) = 0$.
This contradicts the hypothesis that $Z_1$ and the point $0$ can be separated by a polynomial
function.

From now on we consider $h$ to be a polynomial of degree equal to $d$, which vanishes of the first order on $Z$ 
(that is has non identically zero gradient on $Z$) and hence generates the ideal associated to $\partial \G$.
In addition, by possibly enlarging the null set $Z'$, the gradient of $h$ can be assume to be different than zero along $Z$.

Writing $\vec{n_\x}$ as $\nabla h(\x)/\Vert\nabla h(\x)\Vert$, (\ref{aux1}) now reads
\begin{eqnarray}
\nonumber
0&=& \int_Z\langle X,\vec{n}_\x\rangle\,(1-g)\,\x^\alpha d\sigma\,=\,0,\quad \forall\alpha\in\N^n_{2d}\\
\label{aux2}
&=& \int_{Z}\langle X,\nabla h(\x)\rangle\,(1-g)\,\x^\alpha \underbrace{\frac{1}{\Vert\nabla h(\x)\Vert}d\sigma}_{d\sigma'},
\quad\forall \,\alpha\in\N^n_{2d}.
\end{eqnarray}
With $X(\x)=\x$ notice that $\langle \X(\x),\nabla h(\x)\rangle$ is a polynomial of degree at most $d$, and as $\alpha$ in (\ref{aux2}) runs all over $\N^n_{2d}$ we obtain:
$$ \displaystyle\int_{Z}\left[\langle \x,\nabla h(\x)\rangle\,(1-g(\x))\right]^2 d\sigma'= 0,$$
which in turns implies
$$ \left[\langle \x,\nabla h(\x)\rangle\,(1-g(\x))\right]^2 = 0, \quad\mbox{$\sigma'$-a.e. in  $Z$}.$$
As $\langle \x,\nabla h(\x)\rangle\neq0$ for all $\x\in Z$, we find that
$g = 1$ on $Z$ and by continuity on $\partial \G$, in view of the assumption $\G = {\rm int} \,\overline{\G}$.

To complete the poof we now address the uniqueness issue:
Assume that (\ref{thmain-1}) has two distinct solutions $\g_1,\g_2\in\R^{s(d)-1}$. Then
\[\M^d_{2d}(\y)\,\left[\begin{array}{c}0\\ \g_1-\g_2\end{array}\right]\,=\,0.\]
Let $g\in\R[\x]_d$ be the polynomial with coefficient vector $\g_1-\g_2$
so that $g(0)=0$. In view of the assumption $0 \notin \{\x:h(\x)=0\}$ we find $g_1=g_2$.
\end{proof}

Observe that the matrix $\M^d_{2d}(\y)$ contains all moments up to order $3d$ and so
Theorem \ref{thmain} states that it suffices to consider moments up to order $3d$ to recover 
exactly a polynomial $g\in\R[\x]_d$ that is constant on the boundary of $\partial\G$. But of course 
one may sometimes recover $g$ with less moments as exemplified in Example \ref{example-1}
where we only need moments up to order $2d$ (using $\M^d_d(\y)$ instead of $\M^d_{2d}(\y)$).

\begin{corollary}
\label{coro-1}
Let $k\leq 2d$. Under the assumptions of Theorem \ref{thmain}, if 
\begin{equation}
\label{aux11}
\M^d_{k}(\y)\,\left[\begin{array}{c}-1\\ \g\end{array}\right]\,=\,0,\end{equation}
has a unique solution, then $\g$ also solves (\ref{thmain-1}).
\end{corollary}
\begin{proof}
Let $\g$ be the solution of (\ref{thmain-1}), so that $\g$ also solves the system (\ref{aux11}). 
Since (\ref{aux11}) has a unique solution by assumption, the above theorem completes the proof.
\end{proof}
Corollary \ref{coro-1} states that we only need consider 
moments up to order $k+d$ whenever (\ref{aux11}) has a unique solution.
\subsection{Convex supports}
\begin{corollary}
\label{convexity}
Let $\G \subset \R^n$ be a convex bounded open set with real algebraic boundary and $0 \in \G$. Assume  that $\G = {\rm int}\, \overline{\G}$ and that a polynomial of degree at most $d$ 
vanishes on $\partial \G$ (and not at $0$).
Then the system
\[\M^d_d(\y)\,\left[\begin{array}{c}-1\\ \g\end{array}\right]\,=\,0,\]
as a solution $(1,\g)\in\R^{s(d)}$ and the associated polynomial $1-g$ vanishes on the boundary
$\partial\G$. 
\end{corollary}
\begin{proof}
Proceeding as  in the proof of Theorem \ref{thmain}, we know that there exists a solution
$\g^*$ to $\M^d_d(\y)(-1,\g)^T=0$. So let
$\g$ be an arbitrary solution of $\M^d_d(\y)(-1,\g)^T=0$. 
Then again we infer by Stokes Theorem:
\begin{equation}
\label{aux30}
\int_Z\langle X,\vec{n}_\x\rangle\,(1-g)\,\x^\alpha d\sigma\,=\,0,\qquad \forall\alpha\in\N^n_d.
\end{equation}
But then multiplying each side of (\ref{aux30}) with $-g_\alpha$ if $\alpha\neq0$
and with $1$ if $\alpha=0$, and summing up, yields:
\[\int_Z\langle X,\vec{n}_\x\rangle\,(1-g)^2\,d\sigma\,=\,0.\]
Recall that $X(\x)=\x$ and as $\G$ is convex then $\langle \x,\vec{n}_\x\rangle\geq0$ for all $\x\in Z$. Again we may assume that $\langle X,\vec{n}_\x\rangle=0$ on $Z$ only on
a set of zero $\sigma$ measure. Therefore
$g(\x)=1$ for $\sigma$-almost all $\x\in Z$ and by continuity for all $\x\in Z$, and so for all $\x\in\partial\G$
as $\G={\rm int}\,\overline{\G}$.
\end{proof}

 It is important to remark that, in the convex case, the mere knowledge of the moments up to a
certain degree allows us to choose an interior point of the respective set. For instance, the gravity centre
$\x^* = (y_\alpha/y_0)_{|\alpha|=1}$ belongs to the interior of any non-empty, open and bounded convex set $\G$.

\subsection{The singular case} The unfortunate situation when $\x=0$ lies on the Zariski closure of $\partial \G$ can be 
resolved in many ways; for instance by changing the origin of coordinates, or by changing the vector field $X(\x) = \x$ appearing in the proof of the main result above. As for instance:
$$ X(\x) = (\lambda_1 x_1, \lambda_2 x_2, \ldots, \lambda_n x_n)$$
with independent parameters $\lambda_j$ belonging to $\{0,1\}$. In this case the explicit linear system is less symmetrical but
still elementary:
\begin{equation}
\label{aux4}
\int_\G f\,d\x+\int_\G x_j \frac{\partial f}{\partial x_j}\,d\x =0, \ \ 1 \leq j \leq n,\end{equation}
where $f = \x^\alpha(1-g)$ as in the proof of Theorem \ref{thmain}. 
To translate (\ref{aux4}) at the level of moments we need introduce the following
matrices $(\M^{jd}_{2d}(\y))$, $j=1,\ldots,n$, whose rows are indexed by $\alpha\in\N^n_{2d}$
and columns are indexed by $\beta\in\N^n_d$. Their respective entries read:
\begin{equation}
\label{aux5}
\M^{jd}_{2d}(\y)(\alpha,\beta)\,=\,\frac{1+\alpha_j+\beta_j}{1+\alpha_j}\,y_{\alpha+\beta-\e_j},\quad
\alpha\in\N^n_{2d},\,\beta\in\N^n_d
\end{equation}
for every $j=1,\ldots,n$ (where $\e_j=(\delta_{ij})\in\N$). Then (\ref{aux4}) reads
\begin{equation}
\label{aux6}
\M^{jd}_{2d}(\y)(\alpha,\beta)\,\left[\begin{array}{c}-1\\ \g\end{array}\right]\,=\,0,\quad j=1,\ldots,n.
\end{equation}
Hence if $g\in\R[\x]_d$ has a coefficient vector $(0,\g)\in\R^{s(d)}$ such that
$\g$ solves (\ref{aux6}) then
\[\int_Zx_j\,\langle \e_j,\vec{n}_\x\rangle\,(1-g)\,\x^\alpha d\sigma\,=\,0,\qquad \forall\alpha\in\N^n_{2d},\:j=1,\ldots,n.\]

\subsection{Exponentials of polynomials as densities}

So far we have considered only the moment sequence of the Lebesgue measure on $\G$. Without much change we can adapt the preceding calculations to the more general case where the reference measure is $d\mu:=\exp(p(\x))d\x$ for some polynomial $p\in\R[\x]_t$.

Indeed, with $\X(\x)=\x$ and $f=\x^\alpha(1-g)\exp(p)$ Stokes's formula (\ref{eq-stokes}) now reads
\begin{equation}
\label{extension-1}
(n+\vert\alpha\vert) \int_\G \x^\alpha (1-g)\underbrace{\exp(p)d\x}_{d\mu}+\int_\G\langle \x,\nabla p(\x)\rangle \x^\alpha(1-g)\underbrace{\exp(p)\,d\x}_{d\mu}
\end{equation}
\[-\int_\G\langle \x,\nabla g(\x)\rangle \x^\alpha\underbrace{\exp(p)\,d\x}_{d\mu}\,=\,\int_{\partial\G}\langle\x,\vec{n}_\x\rangle 
f\,d\sigma\,=\,0,\]
whenever $g$ vanishes on $\partial\G$.
So let $\y=(\y_\alpha)$, $\alpha\in\N^n$, with:
\[y_\alpha\,:=\,\int \x^\alpha\,d\mu\,=\,\int_\G\x^\alpha\,\exp(p(\x))\,d\x,\qquad\alpha\in\N^n.\]
Then for each $\alpha\in\N^n$, (\ref{extension-1}) 
translates again into a certain linear combination of moments $\y_\beta$ must be zero.
Therefore one may again build up a matrix $\widehat{\M}^d_{k}(\y)$ such that 
(\ref{extension-1}) for all $\alpha\in\N^n_k$ reads:
\begin{equation}
\label{extension-2}
\widehat{\M}^d_{k}(\y)\,\left[\begin{array}{c}-1\\ \g\end{array}\right]\,=\,0.
\end{equation}
The difference is that now this matrix $\widehat{\M^d}_{k}(\y)$ contains moments up to order 
$k+d+t$. 

Theorem \ref{thmain} remains valid if we replace $\M^d_{2d}(\y)$ with $\widehat{\M}^d_{2d}(\y)$.

\subsection{Non algebraic boundary}

Theorem \ref{thmain} suggests a simple strategy to approximately recover the boundary $\partial\G$ when
the latter is not algebraic. By still considering the same moment-like matrix $\M^d_d(\y)$ one may compute the polynomial
$g\in\R[\x]_d$ with coefficient vector $(-1,\g)\in\R^{s(d)}$ such that
$(-1,\g)$ is the right-eigenvector of the matrix $\M^d_d(\y)$ corresponding to the eigenvalue with smallest 
absolute value (if they are all real), or alternatively, the singular vector corresponding to the smallest singular value. We illustrate this strategy on the following simple example.

\begin{example}
Let $\G=\{\x\in\R^2: x_1\geq0;\:x_2\geq1;\:x_2\leq \exp(-x_1+1)\,\}$.
After the change of coordinate $u_1=x_1-1;u_2=x_2$, $\G=\{\u\in\R^2:\,u_1\geq-1;\,u_2\geq1;\,u_2\leq \exp(-u_1)\,\}$
and the origin is not on the boundary $\partial\G$. The shape of $\G$ is displayed in Figure \ref{figure-1}.
\begin{figure}[ht]
\begin{center}
\includegraphics[width=0.6\textwidth]{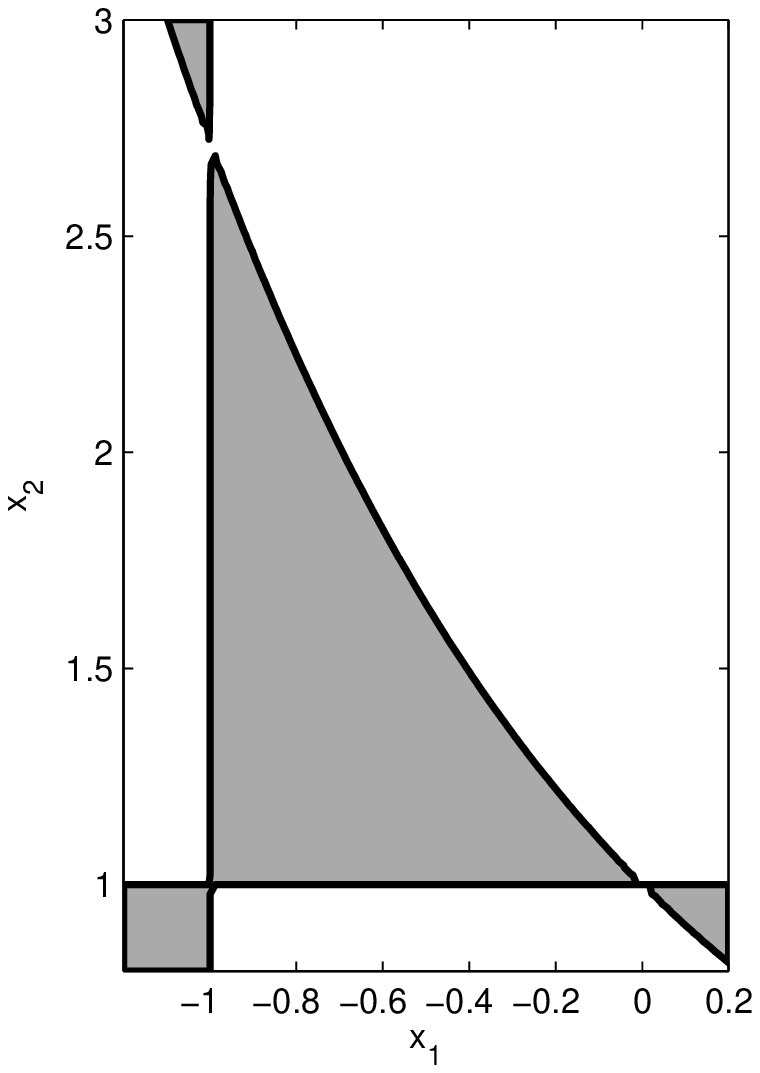}
\caption{Shape of $\G=\{\x: x_1\geq-1;\,x_2\geq1;\,x_2\leq\exp(-x_1)\}$ \label{figure-1}}
\end{center}
\end{figure}
With $d=4$ the all real eigenvalues of $\M_4^4(\y)$ are:
 \[( 2.554403541590561,\,   0.029721326859401,\,   0.004701287525356,\,  -0.001205165323011,\]
 \[  0.000501376438286,\,  0.000034728492891,\,  -0.000014265137533,\,  0.000004783118091,\]
\[  0.000000553859294, \,-0.000000246021037,\,   0.000000011479566,\,   0.000000004666064,\]
\[ -0.000000001994621,   0.000000000000000,\,   0.000000000003563)\]
which shows that one eigenvalue is almost equal to zero (up to Matlab eight digits numerical precision). 
By computing its corresponding right-eigenvector we obtain the following polynomial
\[\u\mapsto g(\u)=
2.554403541590561 +  0.029721326859401\,u_1   +0.004701287525356\,u_2 \]
\[ -0.001205165323011\,u_1^2+   0.000501376438286  \,u_1u_2 +0.000034728492891\,u_2^2\]
\[  -0.000014265137533 \,u_1^3  +0.000004783118091u_1^2u_2 +   0.000000553859294\,u_1u_2^2 \]
\[ -0.000000246021037 \,u_2^3 + 0.000000011479566 \,u_1^4+  0.000000004666064\,u_1^3u_2\]
\[  -0.000000001994621 \,u_1^2u_2^2+   0.000000000003563\,u_2^4\]
and one may see in Figure \ref{figure-2} 
that the compact connected component of the sublevel set $\{\x: g(\x)\leq 0\}$ practically coincides with $\G$!
\begin{figure}[h!]
\begin{center}
\includegraphics[width=0.6\textwidth]{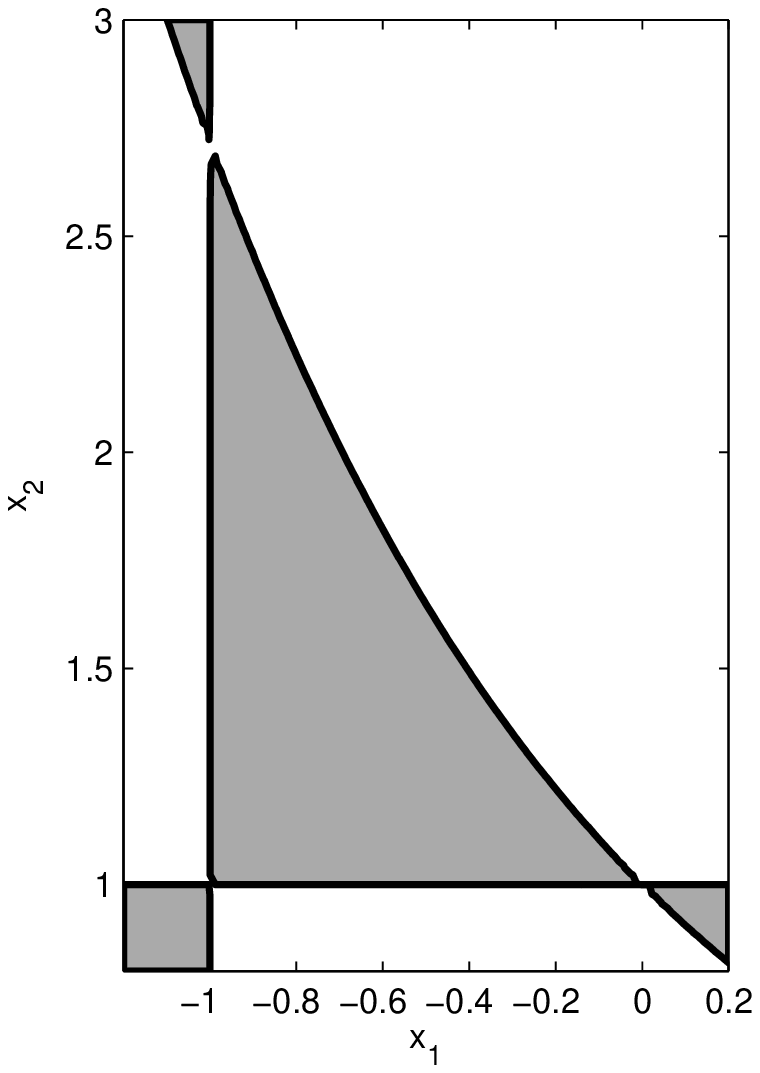}
\caption{Shape $\G'=\{\x: g(\x)\leq0\}$ with $d=4$\label{figure-2}}
\end{center}
\end{figure}
\begin{figure}[h!]
\begin{center}
\includegraphics[width=0.6\textwidth]{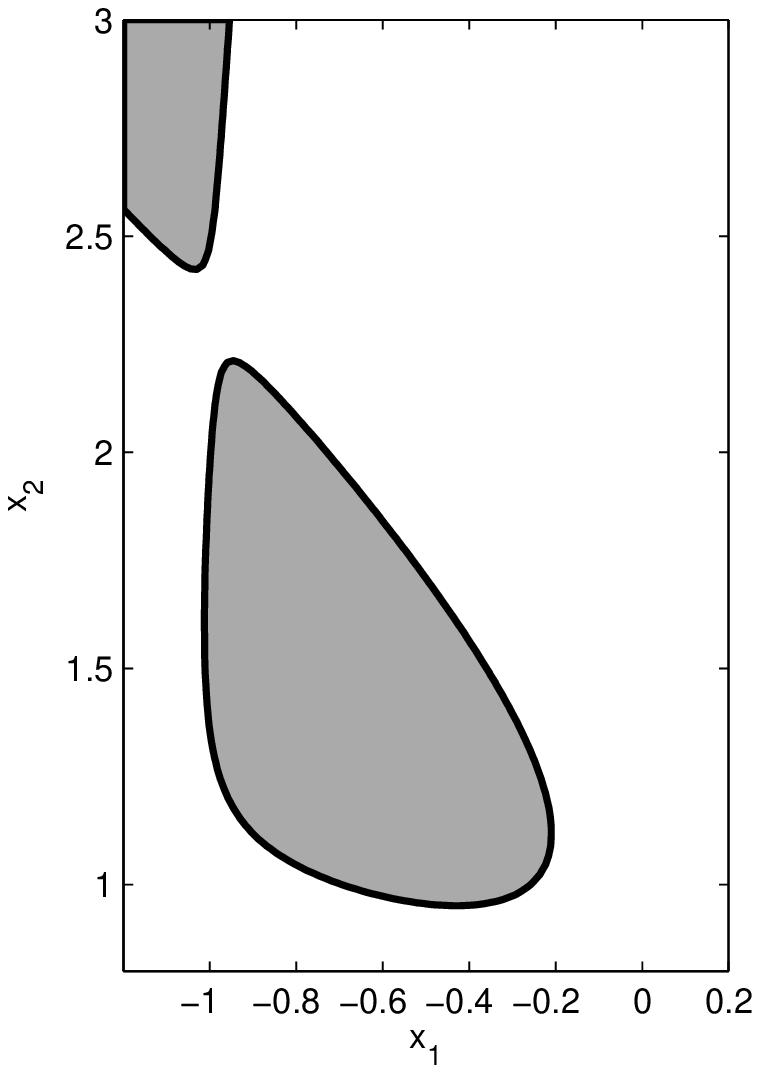}
\caption{Shape $\G'=\{\x: g(\x)\leq0\}$ with $d=3$\label{figure-3}}
\end{center}
\end{figure}
With $d=3$ one obtains the sublevel set displayed in Figure \ref{figure-3} whose compact connected component 
still gives another  (but rough) approximation of $\G$.
 \end{example}

\section{Finite determinateness}

An intriguing conclusion emerges from Theorem \ref{thmain}: the moments of a bounded semi-algebraic set defined by a single polynomial inequality are finitely determined. Specifically, the moments of low degree determine the rest of moments. Except the 1D case $(n=1)$ and a few well studied classes of domains (polyhedra, quadrature domains in 2D, sublevels of homogeneous polynomials)
or weights (Gaussian) the constructive aspects of this determinateness remain unknown.

\subsection{Bounded support} When speaking about finite determinateness, 
Theorem \ref{thmain} has a relevant implication to probability theory. Henceforth we restrict our attention to the particular case when
the underlying set $\G$ is described by a single polynomial inequality.

\begin{thm}
\label{th-proba} Let $g$ be a polynomial of degree $d$, so that the set
$$ \G =\{ \x \in \R^n:\: g(\x) <1\}$$ is bounded and $g(0) \neq 1.$

Then 
the infinite sequence of moments 
$\y=(y_\alpha)$, $\alpha\in\N^n$, of the Lebesgue measure restricted to $\G$, is determined
by its initial section $(y_\alpha)$, $\alpha\in\N^n_{3d}$.

Similarly, all moments 
$\y=(y_\alpha)$, $\alpha\in\N^n$, of the measure $d\mu=\exp(p(\x))d\x$ on $\G$ with $p\in\R[\x]_t$,
are determined
by the finite subset $(y_\alpha)$, $\alpha\in\N^n_{3d+t}$.
\end{thm}
\begin{proof}
Let $\y^*=(y_\alpha)$, $\alpha\in\N^n_{3d}$, be the vector of moments of the Lebesgue measure on $\G$,
and let $1-g\in\R[\x]_d$ be the polynomial in Theorem \ref{th-proba} with vector of coefficients 
$(1,-\g)\in\R^{s(d)}$. Then by Theorem \ref{thmain}, $\g$ solves (\ref{aux1})
which implies that each entry of $\g$ is a function
of $\y^*$ and so we may and will write $\g=(g_\alpha(\y^*))$, $\alpha\in\N^n_d$. But then for every $\beta\in\N^n$,
\[y_\beta\,=\,\int_{\G}\x^\beta\,d\x\,=\,\int_{\{\x:\displaystyle\sum_{\alpha\in\N^n_d} g_\alpha(\y^*)\x^\alpha\leq 1\}}\x^\beta\,d\x
\,=:\,f_\beta(\y^*),\]
is a function of $\y^*$. Same arguments apply for the second statement 
with $d\mu=\exp(p(\x))d\x$.
\end{proof}
In other words, let $\G$ be as in Theorem \ref{th-proba} and suppose that one knows the vector $\y^*$
of moments up to order $3d$ for the Lebesque measure on $\G$. Then one can 
construct the polynomial $g\in\R[\x]_d$ in Theorem \ref{thmain}. All other moments $y_\beta$, $\vert\beta\vert>3d$,
are obtained by integrating $\x^\beta$ on $\G$ which is clearly a function of $\y^*$.\\

\begin{remark} The above determinateness phenomenon is similar to the Gaussian case where $\G=\{\x: g(\x)\leq 1\}$ and $g=\x^T\Sigma\x$
for some positive definite matrix $\Sigma\succ0$. Indeed, 
\[\int_{\G}\x^\beta\,d\x\,=\,\theta_\beta\int_{\R^n}\exp(-\x^T\Sigma\,\x)\,d\x,\]
for some constant $\theta_\beta$ that depends only on the dimension $n$ and $\vert\beta\vert$.
But then
\begin{eqnarray*}
\int_{\G}\x^\beta\,d\x&=&\theta_\beta\int_{\R^n}\x^\beta\,\exp(-\x^T\Sigma\,\x)\,d\x,\qquad\beta\in\N^n,\\
&=&\theta_\beta\int_{\R^n}\x^\beta\,\exp(-\x^T\Delta(\y^*)^{-1}\,\x)\,d\x,\qquad\beta\in\N^n,\\
&=&f_\beta(\y^*),\qquad\beta\in\N^n,
\end{eqnarray*}
where 
\[\Delta(\y^*)\,:=\,\int_{\R^n} \x\x^T\exp(-\x^T\Sigma\x)d\x,\]
i.e., up to scaling, $\Delta(\y^*)$ is the covariance matrix $\Sigma^{-1}$ (i.e. matrix of moments or order $2$).
\end{remark}

\subsection{Exponential weights as densities}
Recall that a function $f:\R^n\to\R$ is said to be quasi-homogeneous if there exists 
$\u\in\Q^n$ such that $f(\lambda^{u_1}x_1,\ldots,\lambda^{u_n}x_n)=\lambda f(\x)$ for all $\x\in\R^n$, and all $\lambda>0$;
then $f$ is also said to be $\u$-quasi-homogeneous or quasi-homogeneous of type $\u$.

Consider the convex cone
\[C\,:=\,\{\,g\in\,\R[\x]_d:\:\int_{\R^n}\exp(-g(\x))\,d\x\,<\,\infty\,\},\]
and let $\mu_g$ be the Borel measure on $\R^n$
with density $\x\mapsto \exp(-g(\x))d\x$ for some
polynomial $g\in C$. As usual write
\[\x\mapsto g(\x)\,=\,\sum_{\alpha\in\N^n_d}g_\alpha\,\x^\alpha\,=\,
\sum_{k=0}^dg_k(\x),\]
where each $g_k\in\R[\x]_k$ is homogeneous of degree $k$.
Let $(g_0,\g)=(g_\alpha)$, $\alpha\in\N^n_d$, 
be the vector of coefficients of $g\in\R[\x]_d$, and let $y_\alpha(\cdot):C\to\R$ be the function
\begin{equation} 
\label{new-y}
\g\mapsto y_\alpha(\g)\,:=\,\int_{\R^n} \x^\alpha \,d\mu_g\,=\,
\int_{\R^n} \x^\alpha \exp(-g(\x))\,d\x\,<\,\infty,\quad\alpha\in\N^n.\end{equation}
\begin{lemma}
\label{lemma-q}
For every $\alpha\in\N^n$, fixed, the function $\g\,\mapsto\,y_\alpha(\g)$
is $\u$-quasi homogeneous where $\u\in\Q^{s(d)}$ and $u_\beta=-\vert\beta\vert/(n+\vert\alpha\vert)$ for all $\beta\in\N^n_d$.
In addition:
\begin{equation}
\label{partial-deriv}
\frac{\partial y_\alpha(\g)}{\partial g_\beta}\,=\,-\int_{\R^n}\x^{\alpha+\beta}\,d\mu_g.
\end{equation}
\end{lemma}
\begin{proof}
Let $\u=(u_\beta)$, $\beta\in\N^n_d$, with $u_\beta�=-\vert\beta\vert/(n+\vert\alpha\vert)$.
Then with $\lambda\in\R$
\begin{eqnarray*}
y_\alpha((\lambda^{u_\beta} g_\beta))&=&\int_{\R^n}\x^\alpha\exp(-\displaystyle\sum_{\beta}g_\beta\lambda^{u_\beta}\x^\beta)\,d\x\\
&=&\int_{\R^n}\x^\alpha\exp(-\displaystyle\sum_{\beta}g_\beta(\lambda^{-1/(n+\vert\alpha\vert)}\x)^\beta)\,d\x\\
&=&\lambda\int_{\R^n}\x^\alpha\exp(-g(\x))\,d\x\,=\,\lambda\,y_\alpha(\g).
\end{eqnarray*}
Finally, (\ref{partial-deriv}) follows from derivation under the integral sign which is justified because of
the exponential weight.
\end{proof}
We are now able to state the main result of this section.
\begin{thm}
\label{th-Rn}
Let $\mu_g$ be as in (\ref{new-y}) with $g\in C$.  Then for each $\alpha\in\N^n$,
\begin{equation}
\label{th-Rn-1}
(n+\vert\alpha\vert)\,\int_{\R^n}\x^\alpha\,d\mu \,=\,
\sum_{k=1}^d k\int_{\R^n}\x^\alpha g_k(\x)\,d\mu,
\end{equation}
or, equivalently:
\begin{equation}
\label{th-Rn-2}
(n+\vert\alpha\vert)\,y_\alpha(\g)\,=\,\sum_{k=1}^d k\sum_{\vert\beta\vert=k}g_\beta\, y_{\alpha+\beta}(\g).
\end{equation}
\end{thm}
\begin{proof}
As $g\mapsto y_\alpha(\g)$ is $\u$-quasi homogeneous, Euler's identity for quasi homogeneous functions yields
\begin{eqnarray*}
y_\alpha(\g)&=&\sum_{\beta\in\N^n_d}u_\beta\, g_\beta\,\frac{\partial y_\alpha(\g)}{\partial g_\beta}\\
&=&\sum_{\beta\in\N^n_d}\frac{\vert\beta\vert}{n+\vert\alpha\vert}g_\beta\,\int_{\R^n}\x^{\alpha+\beta}\,\exp(-g(\x))\,d\x\\
&=&\sum_{k=0}^d\frac{k}{n+\vert\alpha\vert}\,\int_{\R^n}\x^\alpha\,g_k(\x)\,\exp(-g(\x))\,d\x
\end{eqnarray*}
which is the desired result.
\end{proof}
As a corollary we obtain the reconstruction of $g\in\R[\x]_d$ from knowledge of finitely
moments $\y^d=(y_\alpha)$, $\alpha\in\N^n_d$, of $\mu_g$. 
Let $\M_d(\y)$ be the usual $s(d)\times s(d)$ moment matrix of order $d$ associated with $\mu_g$, i.e.,
\[\M_d(\y)(\alpha,\beta)\,=\,y_{\alpha+\beta},\quad\alpha,\beta\,\in\N^n_d.\]
Notice that $\M_d(\y)$ is non singular as $\mu_g$ has a positive density. Next, 
let $\M^d(\y)$ be the $s(d)\times s(d)$ matrix with rows and columns indexed by
$\alpha,\beta\in\N^n_d$ and with entries
\begin{equation}
\label{matrixMd}
\M^d(\y)(\alpha,\beta)\,=\,\left\{
\begin{array}{rl}y_{\alpha+\beta},&\beta=0\\
\frac{\vert\beta\vert\,y_{\alpha+\beta}}{n+\vert\alpha\vert},&0\neq\beta\in\N^n_d\end{array}
\right.,\quad \alpha\in\N^n_d.\end{equation}
Then
(\ref{th-Rn-2}) for all $\alpha\in\N^n_d$ reads
\[\M^d(\y)\,\left[\begin{array}{c}-1\\ \g\end{array}\right]\,=\,0,\]
or, equivalently, using the moment matrix $\M_d(\y)$,
\[\Delta\,\M_d(\y)\,\DD\,\left[\begin{array}{c}-1\\ \g\end{array}\right]\,=\,\Delta_0\,\y^d\]
where $\Delta,\DD$ and $\Delta_0$ are diagonal matrices defined by:
\[\Delta(\alpha,\alpha)\,=\,1/(n+\vert\alpha\vert);\quad
\DD(\alpha,\alpha)\,=\,\left\{\begin{array}{rl}1,&\alpha=0\\ \vert\alpha\vert,&0\neq\alpha\end{array}\right.
\quad\alpha\in\N^n,\]
\[\Delta_0(\alpha,\alpha)\,=\,1-1/(n+\vert\alpha\vert),\quad\alpha\in\N^n_d.\]
\begin{corollary}
\label{coro-exp}
Let $\mu_g$ be the Borel measure in (\ref{new-y}) where $g\in\R[\x]_d$ has coefficient vector
$(g_0,\g)\in\R^{s(d)}$.
Then $\g$ is the unique solution of the linear system
\[\M^d(\y)\,\left[\begin{array}{c}-1\\ \v\end{array}\right]\,=\,0,\]
with $\M^d(\y)$ as in (\ref{matrixMd}) and 
\[g_0=\ln\left(\displaystyle\int_{\R^n}\exp(-\tilde{g}(\x)d\x\right)-\ln y_0,\]
where $\tilde{g}\in\R[\x]_d$
has coefficient vector $(0,\g)$.
\end{corollary}
\begin{proof}
The above linear system  has always the solution $\v=\g$ because
$\M^d(\y)(-1,\g)^T=0$  is just a rephrasing of (\ref{th-Rn-2}) with all $\alpha\in\N^n_d$. 
But this is equivalent to
\[\underbrace{\Delta\,\M_d(\y)\,\DD}_{\Theta(\y)}\,\left[\begin{array}{c}-1\\ \v\end{array}\right]\,=\,\Delta_0\,\y^d.\]
As already noticed, the matrix $\Theta(\y)$ is invertible because 
$\M_d(\y)$ is the moment matrix of $d\mu_g$ which has a 
positive density. And so
\[(-1,\v)^T\,=\,(-1,\g)^T\,=\,\Theta(\y)^{-1}\Delta_0\,\y^d.\]
To obtain the constant coefficient $g_0$, observe that
\[y_0\,=\,\int_{\R^n}\exp(-g(\x))\,d\x\,=\,\exp(-g_0)\int_{\R^n}\exp(-\tilde{g}(\x))\,d\x,\]
which yields the final statement.
\end{proof}
So again and as in the bounded case, one may recover the polynomial $g\in\R[\x]_d$
this time from the knowledge of finitely many moments $\y=(y_\alpha)$, $\alpha\in\N^n_{2d}$
(up to order $2d$). This also implies that all other moments $y_\alpha$ with $\vert\alpha\vert>2d$
are functions of those up to order $2d$.

Observe that the family of measures $\{\mu_g: g\in C\}$ with density $\exp(-g(\x))$ form a so-called {\it exponential family}
(well-studied notably in Probability and Statistics) for which it is known that moments up to order $d$ determine the other ones.
Estimation of a parameter $g\in C$ given i.i.d. observations $\{\v_d(\x(i))\}$, $i=1,\ldots,N$, can be done by solving a convex optimization
problem, e.g. via  maximum entropy methods. However, and to the best of our knowledge,
exact reconstruction results  from moments (even with redundant information) like in Theorem \ref{thmain}
and Corollary \ref{coro-exp} are new. Moreover, notice also that in Theorem \ref{thmain}
the family of measures $\{\mu_g:g\in\R[\x]_d\}$, on $\{\x:g(\x)<1\}$ with uniform or exponential density, do {\it not} form an exponential family.

\section{Examples}

We illustrate the results above with a few low degree and low dimensional examples.

\begin{example}
\label{example-1}
Let us consider the two-dimensional example of the annulus
\[\G\,:=\,\{\,\x\in\R^2:\:1-x_1^2-x_2^2\geq0;\:x_1^2+x_2^2-s\geq0\,\},\quad 0<s<1.\]
That is, $\G$ is the set of points between the two circles $\{\x:1-\Vert\x\Vert^2=0\}$ and 
$\{\x:s-\Vert\x\Vert^2=0\}$ displayed in Figure \ref{annulus}.
\begin{figure}[h!]
\begin{center}
\includegraphics[width=0.6\textwidth]{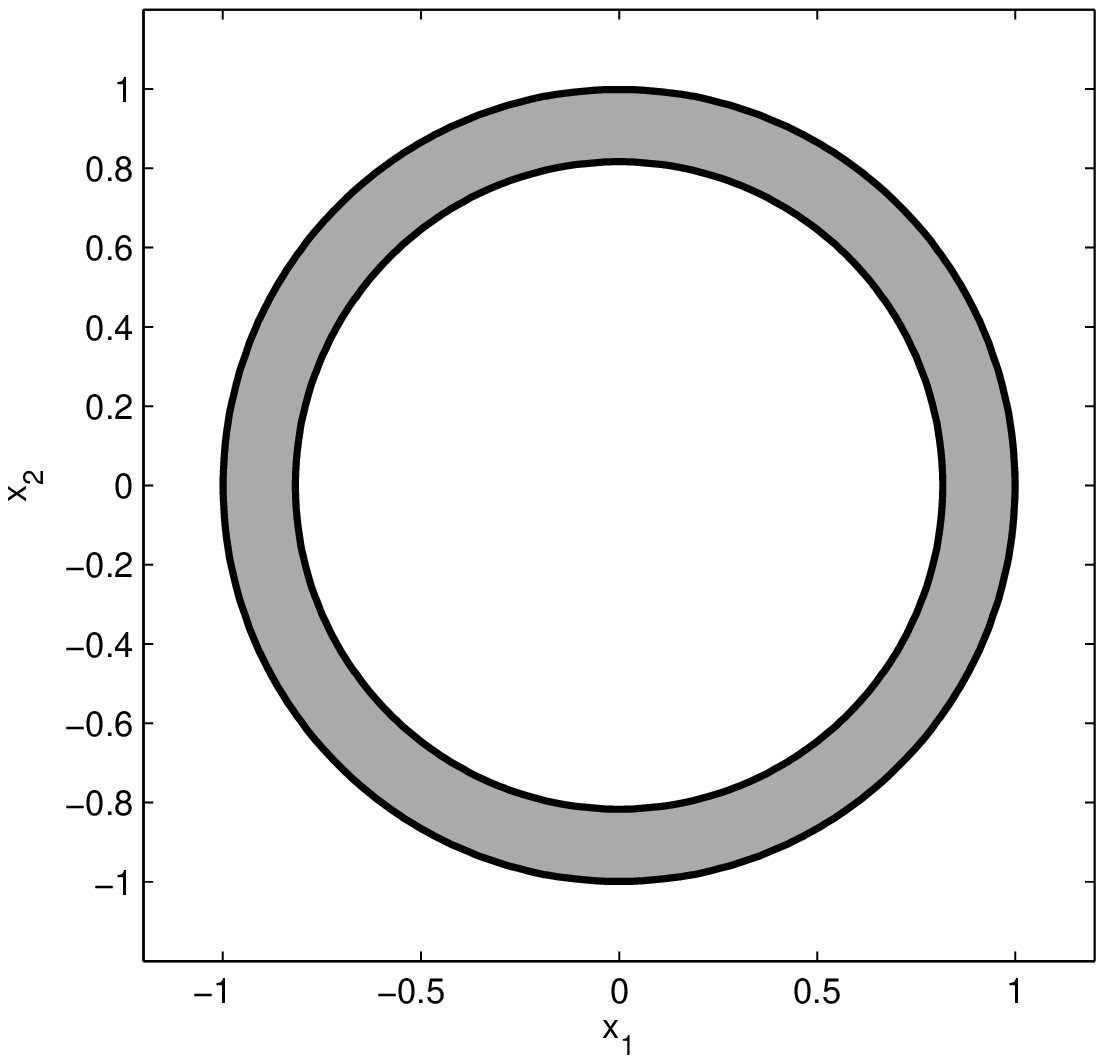}
\caption{The annulus $\G=\{\x: (1-x_1^2-x_2^2)(x_1^2+x_2^2-2/3)\geq0\}$\label{annulus}}
\end{center}
\end{figure}
With $s:=2/3$ and up to a constant, the moment matrix reads:
\[{\tiny\left[\begin{array}{ccccccccccccccc}
 1.0000     &    0   &      0   & 0.2500      &   0 &   0.2500 &        0    &     0    &     0     &    0& 0.125 & 0 &0.0417&0&0.1250\\
         0  &  0.2500  &       0   &      0     &    0       &  0   & 0.1250   &      0  &  0.0417    &     0&  0&0&0&0&0 \\
         0    &     0   & 0.2500   &      0   &      0     &    0   &      0 &   0.0417     &    0&    0.1250&0&0&0&0&0\\
    0.2500    &     0    &     0   & 0.1250   &      0   & 0.0417  &       0    &     0      &   0    &     0& 0.0781&0&0.0156&0&0.0156\\
         0   &      0   &      0   &      0  &  0.0417    &     0       &  0     &    0    &     0   &      0&0&0.0156&0&0.0156&0\\
    0.2500    &     0    &     0  &  0.0417     &    0 &   0.1250 &        0  &       0   &      0   &      0&0.0156&0&0.0156&0&0.0781\\
         0  &  0.1250    &     0     &    0     &    0     &    0 &   0.0781    &     0  &  0.0156 &        0&0&0&0&0&0\\
         0    &     0   & 0.0417  &       0   &      0 &        0   &      0   & 0.0156   &      0 &   0.0156&0&0&0&0&0\\
         0  &  0.0417    &     0 &        0     &    0     &    0 &   0.0156   &      0  &  0.0156   &      0&0&0&0&0&0\\
         0     &    0 &   0.1250  &       0    &     0    &     0   &      0 &   0.0156    &     0 &   0.0781&0&0&0&0&0\\
    0.1250   &      0   &      0 &   0.0781 &        0  &  0.0156 &        0  &       0     &    0     &    0&0.0547&0&0.0078&0&0.0047\\
         0    &     0  &       0     &    0  &  0.0156    &     0    &     0     &    0     &    0     &    0&0&0.0078&0&0.0047&0\\
    0.0417    &     0      &   0   & 0.0156 &        0  &  0.0156   &      0    &     0   &      0  &       0&0.0078&0&0.0047&0&0.0078\\
         0     &    0    &     0   &      0   & 0.0156  &       0     &    0   &      0     &    0  &       0&0&0.0047&0&0.0078&0\\
    0.1250    &     0    &     0  &  0.0156    &     0  &  0.0781    &     0     &    0    &     0     &    0&0.0047&0&0.0078&0&0.0547\end{array}\right]}\]
 and the matrix $15\times 15$ matrix $\M_d^d(\y)\,(=\M^4_4(\y))$ reads:  
    
 \[{\tiny\left[\begin{array}{ccccccccccccccc}
     0.3333    &     0      &   0  &  0.2778    &     0  &  0.2778   &      0       &  0   &      0   &      0 &   0.2639   &      0   & 0.0880&0&0.2639\\
         0  &  0.1852   &      0     &    0    &     0      &   0   & 0.1759  &       0  &  0.0586     &    0     &    0     &    0   &      0&0&0\\
         0   &      0 &   0.1852    &     0     &    0       &  0    &     0 &   0.0586   &      0  &  0.1759    &     0     &    0    &     0&0&0\\
    0.1389   &      0   &      0  &  0.1319   &      0  &  0.0440   &      0     &    0     &    0     &    0   & 0.1254    &     0   & 0.0251&0&0.0251\\
         0    &     0    &     0     &    0    &0.0440   &      0      &   0 &        0  &       0   &      0 &        0  &  0.0251   &      0&0.0251&0\\
    0.1389   &      0     &    0  &  0.0440  &       0 &   0.1319   &      0     &    0   &      0     &    0 &   0.0251   &      0 &   0.0251&0&0.1254\\
         0 &   0.1056    &     0     &    0     &    0    &     0&    0.1003    &     0  &  0.0201   &      0    &     0    &     0      &   0&0&0\\
         0  &       0  &  0.0352    &     0    &     0     &    0     &    0 &   0.0201     &    0  &  0.0201  &       0 &        0   &      0&0&0\\
         0   & 0.0352   &      0   &      0      &   0  &       0  &  0.0201    &     0  &  0.0201    &     0     &    0    &     0     &    0&0&0\\
         0     &    0    &0.1056     &    0   &      0    &     0    &     0   & 0.0201   &      0  &  0.1003    &     0    &     0     &    0&0&0\\
    0.0880     &    0     &    0 &   0.0836  &       0  &  0.0167     &    0    &     0    &     0     &    0   & 0.0791  &       0 &   0.0113&0&0.0068\\
         0    &     0      &   0     &    0   & 0.0167   &      0  &       0      &   0  &       0   &      0   &      0&    0.0113   &      0&0.0068&0\\
    0.0293   &      0     &    0   & 0.0167    &     0 &   0.0167  &       0   &      0     &    0   &      0 &   0.0113   &      0  &  0.0068&0&0.0113\\
         0   &      0    &     0    &     0   & 0.0167   &      0      &   0    &     0     &    0    &     0   &      0 &   0.0068      &   0&0.0113&0\\
    0.0880      &   0     &    0  &  0.0167   &      0   & 0.0836  &       0      &   0   &      0     &    0  &  0.0068     &    0   & 0.0113&0&0.0791\end{array}\right]}\]
From the vector of eigenvalues 
 \[{\small \begin{array}{cccccccc}
 (0.6562&   -0.0595  &  0.1624 &   0.2941 &   0.2941  &  0.0053  &  0  &  0.0628 \end{array}}\]
\[ {\small \begin{array}{cccccccc}
 -0.0021  &  0.0178   & 0.0178  & -0.0063 &  -0.0063 &-0.0007& 0.0045)\end{array}}\]
 of $\M_4^4(\y)$   one can  see that ${\rm rank}\,\M_4^4(\y)=14$. 
 The normalized eigenvector associated with the zero eigenvalue reads
\[{\small \begin{array}{ccccccccccccccc}
  (0.1925 &  0 &   0  & -0.4811  & 0 &  -0.4811  &  0 &  0  &  0 &  0 &   0.2887 &  0   & 0.5774& 0& 0.2887)\end{array}}\]
After scaling, in
\[{\small \begin{array}{ccccccccccccccc}
(0.6667  & 0  &  0   &-1.6667 &  0&   -1.6667 &   0&  0&    0&   0&    1&0& 2&0&1)\end{array}}\]
one recognizes the vector of coefficients of the polynomial
\[\x\mapsto g(\x)=(1-x_1^2-x_2^2)(s-x_1^2-x_2^2)\]
which vanishes on the boundary $\partial\G$.  Finally, one can also check that
the matrix 
\[\M_2^2(\y)\,=\,{\tiny\left[\begin{array}{cccccc}
     0.3333    &     0      &   0  &  0.2778    &     0  &  0.2778 \\
         0  &  0.1852   &      0     &    0    &     0      &   0 \\
         0   &      0 &   0.1852    &     0     &    0       &  0  \\
    0.1389   &      0   &      0  &  0.1319   &      0  &  0.0440 \\
         0    &     0    &     0     &    0    &0.0440   &      0  \\
    0.1389   &      0     &    0  &  0.0440  &       0 &   0.1319  \end{array}\right]}\]
    and the matrix $\M^3_3(\y):$
        \[{\tiny\left[\begin{array}{cccccccccc}
     0.3333    &     0      &   0  &  0.2778    &     0  &  0.2778   &      0       &  0   &      0   &      0\\
         0  &  0.1852   &      0     &    0    &     0      &   0   & 0.1759  &       0  &  0.0586     &    0\\
         0   &      0 &   0.1852    &     0     &    0       &  0    &     0 &   0.0586   &      0  &  0.1759\\
    0.1389   &      0   &      0  &  0.1319   &      0  &  0.0440   &      0     &    0     &    0     &    0 \\
         0    &     0    &     0     &    0    &0.0440   &      0      &   0 &        0  &       0   &      0\\
    0.1389   &      0     &    0  &  0.0440  &       0 &   0.1319   &      0     &    0   &      0     &    0\\
         0 &   0.1056    &     0     &    0     &    0    &     0&    0.1003    &     0  &  0.0201   &      0\\
         0  &       0  &  0.0352    &     0    &     0     &    0     &    0 &   0.0201     &    0  &  0.0201\\
         0   & 0.0352   &      0   &      0      &   0  &       0  &  0.0201    &     0  &  0.0201    &     0 \\
         0     &    0    &0.1056     &    0   &      0    &     0    &     0   & 0.0201   &      0  &  0.1003\end{array}\right]}\]
have respective full rank $6$ and $10$ so that (\ref{thmain-1}) has no solution when $d=2$ or $d=3$.

So this example illustrates Corollary \ref{coro-1} to show that sometimes we only  need consider moments up to order $2d=8$
and not $3d=12$.
\end{example}   
   \begin{example}
   \label{simplex}
  The following (convex) example illustrates that the assumption $0\in{\rm int}\,\G$ is important.
  Let $\G\subset\R^2$ be the two-dimensional simplex $\{\x: x_1+x_2\leq 1;\,\x\geq0\}$.
   The matrix $\M_1^1(\y)$ which reads:
   \[\M^1_1(\y)\,=\,\left[\begin{array}{ccc}
   1/2&   1/4&   1/4\\
   1/6 &1/9 &1/18\\
   1/6 &1/18 &1/9
      \end{array}\right]\]
   has rank $2$ with zero eigenvector $(-1, 1, 1)$. And indeed
   even though the polynomial  $\x\mapsto 1-(x_1+x_2)$ vanishes only on some part $\Omega\subset\partial\G$ of the boundary
   $\partial\G$, (\ref{thmain-1}) holds because $\langle \x,\vec{n}_\x\rangle$ vanishes on $\partial\G\setminus\Omega$.

   What is more surprising is that $\M_2^2(\y)$ has only rank $3$ with 
   three zero-eigenvalues $\lambda_1=\lambda_2=\lambda_3=0$. 
     One has multiplicity $1$ with eigenvector
   \[(1, -4.405781742297638, 0.823092738895580, 3.405781742297638,\]
   \[ 1.582689003402045, -1.823092738895576),\]
   whereas the other eigenvalue is double with associated eigenvector
   \[(1, 0.744634776919192, -6.597713889514154, -1.744634776919187, \]
   \[3.853079112594940, 5.597713889514163).\]
   One can check that the two associated polynomials vanish when $x_1+x_2=1$.
   
     Similarly $\M^3_3(\y)$ has only rank $7$ with three zero-eigenvalues whose associated
   eigenvectors are polynomials of degree 2 which vanish whenever $x_1+x_2=1$.

   Even more surprising is that $\M^4_4(\y)$ has rank $14$ with associated zero-eigenvector \\
$(-1,1,1,0,\ldots,0)$.

    \end{example}

\end{document}